\documentclass[a4paper]{article}
\usepackage[affil-it]{authblk}
\usepackage{graphicx}
\usepackage{graphicx}
\usepackage{cite}
\usepackage{float}
\usepackage{amsmath}
\usepackage{amsthm}
\usepackage{amsfonts}
\usepackage{amssymb,latexsym}
\newtheorem{theorem}{Theorem}

\newtheorem{remark}{Remark}[section]

\newtheorem{example}{Example}[section]

\begin{document}

\title{Directional differentiability, coexhausters, codifferentials and polyhedral DC functions.}

\author{Majid E. Abbasov         
}


\affil{
              St. Petersburg State University, SPbSU,\\ 7/9 Universitetskaya nab., St.
              Petersburg, 199034 Russia. \\

              Institute for Problems in Mechanical
              Engineering of the RAS\\
              61, Bolshoj pr. V.O., St. Petersburg,
199178\\
              {m.abbasov@spbu.ru, abbasov.majid@gmail.com}           
}

\maketitle

\begin{abstract}

Codifferentials and coexhausters are used to describe
nonhomogeneous approximations of a nonsmooth function. Despite the
fact that coexhausters are modern generalizations of
codifferentials, the theories of these two concepts continue to
develop simultaneously. Moreover, codifferentials and coexhausters
are strongly connected with DC functions. In this paper we trace
analogies between all these objects, and prove the equivalence of
the boundedness and optimality conditions described in terms of
these notions. This allows one to extend the results derived in
terms of one object to the problems stated via the other one.
Another contribution of this paper is the study of connection
between nonhomogeneous approximations and directional derivatives
and formulate optimality conditions in terms of nonhomogeneous
approximations.\end{abstract}


\section*{Introduction}
\label{intro_abbasov} Among the variety of approaches of nonsmooth
analysis \cite{rock70} the method of quasidifferential stands out
due to its constructiveness. One important advantage of this
approach is that all the tools and methods can be built and used
not only theoretically but also in practical problems. The
approach goes back to the early 80-th
when Demyanov, Rubinov and Polyakova proposed and studied the
notion of quasidifferentials
\cite{Demyanov-Rubinov-Polyakova_quasidif79,Demyanov-Rubinov_quasidif80,Demyanov-Vasiliev85,Demyanov-Polyakova80}.
Quasidifferentials are pairs of convex compact sets that enable
one to represent the directional derivative of a function at a
point in a form of sum of maximum and minimum of a linear
functions. Quasidifferentials enjoy full calculus, that grants the
calculation of quasidifferentials for a rich variety of functions.
Such functions are also called quasidifferentiable.

Polyakova and Demyanov derived optimality conditions in terms of
these objects and also showed how to find the directions of
steepest descent and ascent when these conditions are not
satisfied. This paved a way for constructing new optimization
algorithms. 
An interesting example of the application of quasidifferential
calculus can be found in \cite{Sukhorukova2017}, where the authors
use this tool to solve  a complex optimization problem   appearing
in the area of Chebyshev approximation.

In some cases, however, these algorithms experience convergence
problem \cite{Demyanov-Malozemov}. This happens due to the fact
that quasidifferentiable set-valued mapping is not continuous in
Hausdorff metric. Similar results have been reported with
exhausters
\cite{Demyanov-Rubinov2001,Demyanov_optimization_99,Demyanov-dem00a,Demyanov-Abbasov_JOTA_13,Demyanov-optimization-2012,Demyanov-dros05,Demyanov-dros06,Demyanov-dros08,Demyanov-Abbasov_IMMO10}
which can be viewed as a generalization of quasidifferentials
\cite{Abbasov_JOTA_17}.

To overcome this drawback Demyanov and Rubinov in the mid 90-th
introduced the notion of codifferentials
\cite{Demyanov_Rubinov_1995}. Codifferential is a pair of convex
compact sets that provides the representation of the approximation
of the studied function in a neighborhood of a given point in the
form of sum of minimum and maximum of affine functions.

Coexhausters arose as a generalization of codifferentials
\cite{Abankin98}. A class of coexhausterable functions is wider
than the class of codifferentiable functions. Coexhausters are
families of convex compact sets which are used to represent the
approximation of a considered function in a neighborhood of a
point as a sum of MaxMin or MinMax of affine functions. The
formulas of calculus for codifferentials and coexhausters have
been derived as well as optimality conditions in terms of these
tools
\cite{Demyanov-Abbasov_JOGO_2013,Abbasov_JOTA16,Abbasov_Vestnik19}.

The usage of continuously codifferentiable and coexhausterable
functions guaranteed stability and convergence of numerical
algorithms, but positive homogeneity property was lost in this
path. In this paper we address this issue by study the optimality
conditions in terms of inhomogeneous approximations.

It must be noted that codifferentials and coexhausters have strong
connection with DC functions
\cite{Aleksandrov1950,Hartman1959,Pardalos02,Bagirov-Ugon_16,Bagirov-Ugon_18}.
Therefore the problem of studying the connection between all these
notions is of high interest. It can enable us to extend results
derived in terms of one object to the problems stated via another
one.

The paper is organized as follows. In Section 1 we establish
connection between directional derivatives and nonhomogeneous
approximations of a function. Then we give definitions of
codifferentials and coexhausters and connect these notions with
the class of difference of polyhedral convex functions. In Section
2 we present Polyakova's (see
\cite{Polyakova_JOGO_2011}) boundedness condition 
in terms of codifferentials. We prove that this condition is
equivalent to the condition of boundedness stated in
\cite{Abbasov_arXiv_2020} in terms of coexhausters. In Section 3
we describe Demyanov's optimality conditions in terms of
coexhausters and Polyakova's optimality conditions in terms of
codifferentials. We demonstrate that these conditions are
equivalent. All the presented results are also considered from DC
functions point of view.

\section{Directional differentiability, codifferentials, coexhausters and polyhedral DC funcions}\label{s190411}

Let a function $f\colon \mathbb{R}^n \rightarrow  \mathbb{R}$ be
given. The function $f$ is called directionally differentiable at
a point $x \in \mathbb{R}^n$ if for every $\Delta \in
\mathbb{R}^{n}$ there exists the final limit
\begin{equation*}\label{Demyanov_Abbasov_eq13appr}
f^{\prime}(x,\Delta) = \lim_{\alpha\downarrow 0} \frac{f(x+\alpha
\Delta)-f(x)} {\alpha}.
\end{equation*}

The value $f^{\prime}(x,\Delta)$ is called the directional
derivative of the function $f$ at the point $x\in \mathbb{R}^n$ in
the direction $\Delta\in \mathbb{R}^n$. Directional derivative
allows us to formulate necessary conditions for a minimum and
maximum (see \cite{rock70}).

\begin{theorem}\label{Demyanov_Abbasov_T2.2}
Let a function $f\colon \mathbb{R}^{n}\to \mathbb{R}$ be
directionally directionally differentiable at a point $x_{\ast}\in
\mathbb{R}^n$. For the point $x_{\ast}$ to be a minimizer of the
function $f$ on $\mathbb{R}^{n}$ it is necessary that
\begin{equation}\label{Demyanov_Abbasov_eq16appr}
f^{\prime}(x_{\ast},\Delta)\geq 0 \quad \forall \Delta \in
\mathbb{R}^{n}.
\end{equation}
\end{theorem}

\begin{theorem}\label{Demyanov_Abbasov_T2.3}
Let a function $f\colon \mathbb{R}^{n}\to \mathbb{R}$ be
directionally directionally differentiable at a point $x^{\ast}\in
\mathbb{R}^n$. For the point $x^{\ast}$ to be a maximizer of the
function $f$ on $\mathbb{R}^{n}$ it is necessary that
\begin{equation}\label{Demyanov_Abbasov_eq20appr}
f^{\prime}(x^{\ast},\Delta)\leq 0 \quad \forall \Delta  \in
\mathbb{R}^{n}.
\end{equation}
\end{theorem}

A point $x_{\ast}$ satisfying condition
(\ref{Demyanov_Abbasov_eq16appr}), is called an $\inf$-stationary
point of the function $f$. A point $x^{\ast}$ satisfying condition
(\ref{Demyanov_Abbasov_eq20appr}), is called a $\sup$-stationary
point of the function $f$.

Let the following expansion holds
\begin{equation}\label{abbasov_expansion_main}
f(x+\Delta)=f(x)+h_x(\Delta)+o_{x}(\Delta)\quad \forall \Delta\in
{\mathbb{R}}^n,
\end{equation}
for a continuous directionally differentiable function $f$, where
\begin{equation*}\label{1904e11a}
\lim_{\alpha\downarrow0}\frac{o_{x}(\alpha\Delta)}{\alpha}=0\quad
\forall\Delta\in{\mathbb{R}}^n.
\end{equation*}
Due to the continuity of $f$ we have $h_x(0_n)=0$ at any point
$x$. Therefore $h_x(\Delta)$ is directionally differentiable at
the origin and
$$f^{\prime}(x,\Delta)=h^{\prime}_x(0_n,\Delta).$$

If $h_x(\Delta)$ is positively homogenous as a function of
$\Delta$ we have
$$h^{\prime}_x(0_n,\Delta)=h_x(\Delta)\quad\forall\Delta\in\mathbb{R}^n,$$
i.e. in this case we can state results similar to Theorems
\ref{Demyanov_Abbasov_T2.2} and \ref{Demyanov_Abbasov_T2.3} by
replacing $f^{\prime}(x,\Delta)$ with $h_{x}(\Delta)$.

If $h_x(\Delta)$ is not positively homogenous then
$$h^{\prime}_x(0_n,\Delta)\neq h_x(\Delta),$$
but inequality $h_x(\Delta)\leq 0$ implies
$h^{\prime}_x(0_n,\Delta)\leq 0$ while inequality $h_x(\Delta)\geq
0$ implies $h^{\prime}_x(0_n,\Delta)\geq 0$. Therefore in this
case the condition
\begin{equation*}\label{abbasov_generel_min_cond}
h_{x_\ast}(\Delta)\geq 0\quad\forall \Delta\in {\mathbb{R}}^n.
\end{equation*}
is sufficient for $x_\ast$ to be an $inf$-stationary point, while
the condition
\begin{equation*}\label{abbasov_generel_max_cond}
h_{x^\ast}(\Delta)\leq 0\quad\forall \Delta\in {\mathbb{R}}^n.
\end{equation*}
is sufficient for $x^\ast$ to be a $sup$-stationary point.

Having a specific form of the approximation $h_{x}(\Delta)$, we
can describe conditions of minimum in terms of objects that define
the form and use these conditions to construct optimization
algorithms. This is the case in smooth case where $h_{x}(\Delta)$
can be presented as the inner product of the gradient and the
direction $h_{x}(\Delta)=\langle\nabla f(x),\Delta\rangle$. For a
nonsmooth function a linear approximation is not applicable and
one have to work with more complicated forms.

The function $f\colon X\to \mathbb{R}$ is called codifferentiable
at a point $x$, if there exist convex compact sets
$\underline{d}f(x)\subset{\mathbb{R}}^{n+1}$ and
$\overline{d}f(x)\subset{\mathbb{R}}^{n+1}$ such that
\begin{equation}\label{1904e1}h_x(\Delta)=
\max_{[a,v]\in{\underline{d}f(x)}}[a+\langle
v,\Delta\rangle]+\min_{[b,w]\in{\overline{d}f(x)}} [b+\langle
w,\Delta\rangle],
\end{equation}
where

  The pair
$Df(x)=[\underline{d}f(x),\overline{d}f(x)]$ is called a
codifferential of the function $f$ at the point $x$. Recall that a
codifferential is a pair of sets in the space $\mathbb{R}^{n+1}$.

The function $f$ is  continuous, therefore from (\ref{1904e1})
(for $\Delta=0_n$) it follows that
 \begin{equation}\label{1904e3am_codiff}
\max_{[a,v]\in{\underline{d}f(x)}}a+\min_{[b,w]\in{\overline{d}f(x)}}
b = 0.
\end{equation}
Since a codifferential function is not uniquely defined at a
point, without loss of generality we can rewrite equality
(\ref{1904e3am_codiff}) as
 \begin{equation}\label{1904e3am_codiff_rewritten}
\max_{[a,v]\in{\underline{d}f(x)}}a=\min_{[b,w]\in{\overline{d}f(x)}}
b = 0.
\end{equation}

A function $f$ is called continuously codifferentiable at a point
$x$ if it is codifferentiable in some neighborhood of the point
$x$ and there exists a codifferential mapping
$Df(x)=[\underline{d}f(x),\overline{d}f(x)]$ which is continuous
in the Hausdorff metric at the point $x$.

Polyhedral codifferential is of high importance for many
applications and therefore we concentrate on this case in the rest
of the paper, i.e.
$$\underline{d}f(x)=\operatorname{co}\left\{[a_i,v_i]\mid i\in
I\right\},\ \overline{d}f(x)=\operatorname{co}\left\{[b_j,w_j]\mid
i\in J\right\},$$ where $I$ and $J$ are finite index sets.

Expression (\ref{1904e1}) implies
\begin{equation}\label{1904e2}
\begin{split}
h_x(\Delta)&=
\max_{[a,v]\in{\underline{d}f(x)}}\min_{[b,w]\in{\overline{d}f(x)}]}[a+b+\langle
v+w,\Delta\rangle]\\
 &=\max_{C\in\underline{E}(x)}\min_{[b,w]\in{C}}[b+\langle w,\Delta\rangle],
\end{split}
\end{equation}
where
$$\underline{E}(x)=\{C\subset{\mathbb{R}}^{n+1}|C=[a,v]+\overline{d}f(x)\},\ [a,v]\in\underline{d}f(x)\}.$$
Similarly we can get the representation
\begin{equation}\label{1904e3}
\begin{split}
h_x(\Delta)&=
\min_{[b,w]\in{\overline{d}f(x)}}\max_{[a,v]\in{\underline{d}f(x)}}[a+b+\langle
v+w,\Delta\rangle]\\&=
\min_{C\in{\overline{E}(x)}}\max_{[a,v]\in{C}}[a+\langle
v,\Delta\rangle],
\end{split}
\end{equation}
where
$$\overline{E}(x)=\{C\subset{\mathbb{R}}^{n+1}|C=[b,w]+\underline{d}f(x)\},\ [b,w]\in\overline{d}f(x)\}.$$
The functions
$$\max_{C\in\underline{E}(x)}\min_{[b,w]\in{C}}[b+\langle w,\Delta\rangle]\quad \mbox{and} \quad \min_{C\in{\overline{E}(x)}}\max_{[a,v]\in{C}}[a+\langle v,\Delta\rangle]$$
represent approximations of the increment of the function $f$ in a
neigbourhood of $x$. The usage of continuously codifferentiable
functions introduced above allows one to guarantee stability and
convergence of numerical algorithms.

The notion of codifferential was introduced in
\cite{Demyanov_Rubinov_1995} where necessary optimality conditions
were stated.
Via expansions (\ref{1904e2}) and (\ref{1904e3}) we obtain the following generalization of the codifferential notion.\\

Let a function $f$ be continuous at a point $x\in X$. We say that
at the point $x$ the function $f$ has an upper coexhauster 
if the following expansion holds:
\begin{equation}\label{1904e3a}
h_x(\Delta)=\min_{C\in{\overline{E}(x)}}\max_{[a,v]\in{C}}[a+\langle
v,\Delta\rangle],
\end{equation}
where $\overline{E}(x)$ is a family of convex compact sets in
${\mathbb{R}}^{n+1}$. The set $\overline{E}(x)$ is called an upper
coexhauster of $f$ at the point $x$.

We say that at the point $x$ the function $f$ has a lower
coexhauster if the following expansion holds:
\begin{equation}\label{1904e3b}h_x(\Delta)=\max_{C\in{\underline{E}(x)}}\min_{[b,w]\in{C}}[b+\langle w,\Delta\rangle],
\end{equation}
where $\underline{E}(x)$ is a family of convex compact sets in
${\mathbb{R}}^{n+1}$. The set $\underline{E}(x)$ is called a lower
coexhauster of the function $f$ at the point $x$.

The function $f$ is  continuous, therefore from (\ref{1904e3a})
and (\ref{1904e3b}) we have
 \begin{equation}\label{1904e3am}
\min_{C\in{\overline{E}(x)}}\max_{[a,v]\in{C}}a =
\max_{C\in{\underline{E}(x)}}\min_{[b,w]\in{C}}b = 0.
\end{equation}
The notion of coexhauster was introduced in
\cite{Demyanov_optimization_99,Demyanov-dem00a}. Similar to the
case of codifferentiable functions, we can consider continuous
upper and lower coexhauster mappings.

It is important to notice that DC functions are codifferentiable
and have upper and lower coexhausters. This means that DC
functions can be studied via the rich theory of codifferentials
and coexhausters. Let us illustrate this.

Local approximation for many DC function can be presented as the
difference of polyhedral convex functions, i.e. in the form
\begin{equation}\label{loc_app_DPC_func}
h_x(\Delta)=\max_{i\in I}[a_i+\langle
v_i,\Delta\rangle]-\max_{j\in J}[b_j+\langle w_j,\Delta\rangle],
\end{equation}
where $I$ and $J$ are finite index sets. We can rewrite
(\ref{loc_app_DPC_func}) in the form

\begin{equation}\label{loc_app_DPC_func_reform}
\begin{split}
h_x(\Delta)&=\max_{i\in I}[a_i+\langle
v_i,\Delta\rangle]+\min_{j\in J}[-b_j-\langle
w_j,\Delta\rangle]\\
&=\max_{[a,v]\in{\underline{d}h}}[a+\langle
v,\Delta\rangle]+\min_{[b,w]\in{\overline{d}h}} [b+\langle
w,\Delta\rangle]\\
&=\max_{C\in{\underline{E}}}\min_{[a,v]\in{C}}[b+\langle
w,\Delta\rangle]\\
&=\min_{C\in{\overline{E}}}\max_{[a,v]\in{C}}[a+\langle
v,\Delta\rangle],
\end{split}
\end{equation}
where $\underline{d}h=\operatorname{co}\left\{[a_i,v_i]\mid i\in
I\right\}$, $\overline{d}h=\operatorname{co}\left\{[-b_j,-w_j]\mid
i\in J\right\}$ and
$$\underline{E}=\left\{\operatorname{co}\{[a_i-b_j,v_i-w_j], j\in J \}\mid  i\in I \right\},$$
$$\overline{E}=\left\{\operatorname{co}\{[a_i-b_j,v_i-w_j], i\in I\}\mid  j\in J \right\}.$$

For the sake of shortness we will use notation $h(\Delta)$ instead
of $h_x(\Delta)$ in what follows.

\section{Boundedness conditions}\label{abbasov_codif_coex_DC_func_sect2}

Local approximation $h$ is often used for construction of
optimization algorithms. Therefore, in the minimization problems,
it is essential that the approximation is bounded from below.
Polyakova derived this condition in terms of codifferential in
\cite{Polyakova_JOGO_2011}.

\begin{theorem}[Polyakova]\label{abbasov_bound_cond_codif}
For the function
$$h(\Delta)=\max_{i\in I}[a_i+\langle
v_i,\Delta\rangle]+\min_{j\in J}[b_j+\langle w_j,\Delta\rangle],$$
where $I$ and $J$ are finite index sets, to be bounded from below
it is necessary and sufficient that for any $j\in J$ the condition
$$-w_j\in\operatorname{co}\left\{v_i\mid i\in I\right\}$$
holds.
\end{theorem}
The same condition was obtained in \cite{Abbasov_arXiv_2020} in
terms of coexhausters.

\begin{theorem}[Abbasov]\label{abbasov_bound_cond_coex}
For the function
$$h(\Delta)=\min_{C\in{\overline{E}}}\max_{[a,v]\in{C}}[a+\langle
v,\Delta\rangle]$$ to be bounded from below it is necessary and
sufficient that the condition
\begin{equation*}\label{abbasov_th1_condition} C\bigcap L\neq\emptyset \quad \forall C\in\overline{E},\end{equation*}
is satisfied, where $L=\left\{(a,0_n)\mid a\in\mathbb{R}\right\}$.
\end{theorem}

Based on Theorems \ref{abbasov_bound_cond_codif} and
\ref{abbasov_bound_cond_coex} we can state and prove general
result which connects polyhedral DC-functions, codifferentials and
coexhausters.

\begin{theorem}\label{abbasov_bound_cond_codif_coex_equiv}
For the function $$h(\Delta)=\max_{i\in I}[a_i+\langle
v_i,\Delta\rangle]-\max_{j\in J}[b_j+\langle w_j,\Delta\rangle],$$
where $I$ and $J$ are finite index sets, to be bounded from below
it is necessary and sufficient that one of the following
equivalent conditions hold
\begin{equation}\label{abbasov_bound_cond_codif_coex_equiv_eq1}
w_j\in\operatorname{co}\left\{v_i\mid i\in I\right\}, \quad
\forall j\in J
\end{equation}
or
\begin{equation}\label{abbasov_bound_cond_codif_coex_equiv_eq2}
C\bigcap L\neq\emptyset \quad \forall C\in\overline{E}
\end{equation}
where $L=\left\{(a,0_n)\mid a\in\mathbb{R} \right\}$ and
$$\overline{E}=\left\{C\mid C=\operatorname{co}\{[a_i-b_j,v_i-w_j],
i\in I\}, j\in J \right\}.$$
\end{theorem}

\begin{proof}
Since the function $h$ can be rewritten in the form
$$h(\Delta)=\max_{i\in I}[a_i+\langle
v_i,\Delta\rangle]+\min_{j\in J}[-b_j-\langle
w_j,\Delta\rangle],$$ boundedness of $h$ yields immediately from
Theorems \ref{abbasov_bound_cond_codif} and
\ref{abbasov_bound_cond_coex}.

Prove that conditions
(\ref{abbasov_bound_cond_codif_coex_equiv_eq1}) and
(\ref{abbasov_bound_cond_codif_coex_equiv_eq2}) are equivalent.
Let condition (\ref{abbasov_bound_cond_codif_coex_equiv_eq2})
holds. Then we have
$$0_n\in\operatorname{co}\{v_i-w_j\mid i\in I\}\quad \forall j\in J.$$
This implies that for an arbitrary $j\in J$ there exists
$\lambda_{i}$, $i\in I$ such that
$$
\begin{cases}
\displaystyle\sum_{i\in I}\lambda_{i}=1,\\
\lambda_{i}\geq 0\quad\forall i\in I,\\
\end{cases}
$$
for which holds the condition
$$\displaystyle\sum_{i\in I}\lambda_{i}(v_i-w_j)=0.$$
Therefore for we have
$$\displaystyle\sum_{i\in I}\lambda_{i}v_i=w_j,$$
what implies (\ref{abbasov_bound_cond_codif_coex_equiv_eq1}).

To prove that (\ref{abbasov_bound_cond_codif_coex_equiv_eq2})
follows from (\ref{abbasov_bound_cond_codif_coex_equiv_eq1}) we
can run the same proof backwards. \qed
\end{proof}

Similar theorem can be stated for the upper boundedness
conditions.

\begin{theorem}\label{abbasov_up_bound_cond_codif_coex_equiv}
For the function $$h(\Delta)=\max_{i\in I}[a_i+\langle
v_i,\Delta\rangle]-\max_{j\in J}[b_j+\langle w_j,\Delta\rangle],$$
where $I$ and $J$ are finite index sets, to be upper bounded it is
necessary and sufficient that one of the following equivalent
conditions hold

\begin{equation}\label{abbasov_up_bound_cond_codif_coex_equiv_eq1}
v_i\in\operatorname{co}\left\{w_j\mid j\in J\right\}, \quad
\forall i\in I
\end{equation}
or
\begin{equation}\label{abbasov_up_bound_cond_codif_coex_equiv_eq2}
C\bigcap L\neq\emptyset \quad \forall C\in\underline{E}
\end{equation}
where $L=\left\{(a,0_n)\mid a\in\mathbb{R} \right\}$ and
$$\underline{E}=\left\{C\mid C=\operatorname{co}\{[a_i-b_j,v_i-w_j],
j\in J\}, i\in I \right\}.$$
\end{theorem}

Now let us demonstrate how these results works.

\begin{example}

Consider the function
$$h(\Delta)=\max\{2\Delta-4,0,-2\Delta-4\}-\max\{\Delta-1,0,-\Delta-1\}.$$

Fig. \ref{exmpl_1_exmpl1_func_graph} shows that $h$ is bounded
from below.
\begin{figure}[H]
\includegraphics[width=0.65\textwidth]{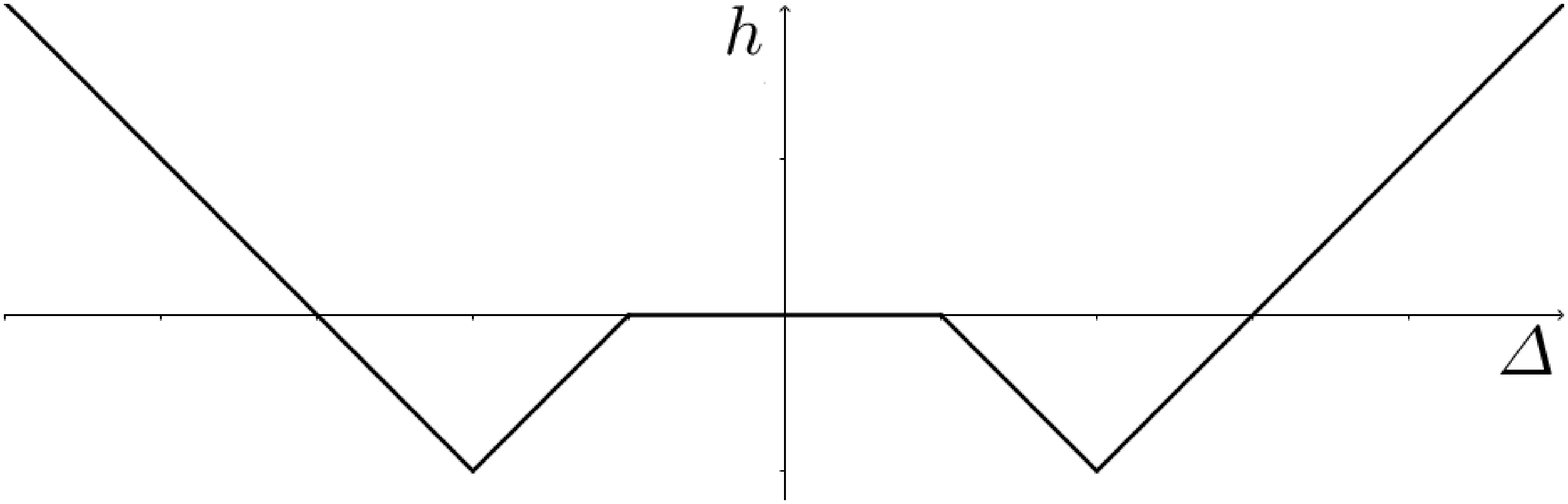}
\caption{Graph of the function $h$ in Example 1.}
\label{exmpl_1_exmpl1_func_graph}
\end{figure}

Since $h$ can be rewritten as
$$h(\Delta)=\max\{2\Delta-4,0,-2\Delta-4\}+\min\{-\Delta+1,0,\Delta+1\},$$
the function has a codifferential of the form (see Fig.
\ref{exmpl_1_exmpl1_func_codif_exstr} a)
$$\underline{d}h=\operatorname{co}\left\{\begin{pmatrix} -&4\\ &2 \end{pmatrix},\begin{pmatrix} 0\\ 0 \end{pmatrix},\begin{pmatrix} -4\\ -2 \end{pmatrix},\right\}, \quad \overline{d}h=\operatorname{co}\left\{\begin{pmatrix} &1\\ -&1 \end{pmatrix},\begin{pmatrix} 0\\ 0 \end{pmatrix},\begin{pmatrix} 1\\ 1 \end{pmatrix}\right\},$$
and an upper coexhauster $\overline{E}=\{C_1,C_2,C_3\}$, where
$$C_1=\operatorname{co}\left\{\begin{pmatrix} -&3\\ &1 \end{pmatrix},\begin{pmatrix} &1\\ -&1 \end{pmatrix},\begin{pmatrix} -3\\ -3 \end{pmatrix}\right\},\quad C_2=\operatorname{co}\left\{\begin{pmatrix} -&4\\
&2
\end{pmatrix},\begin{pmatrix} 0\\ 0 \end{pmatrix},\begin{pmatrix}
-4\\ -2 \end{pmatrix}\right\},$$
$$C_3=\operatorname{co}\left\{\begin{pmatrix} -&3\\ &3 \end{pmatrix},\begin{pmatrix} 1\\ 1 \end{pmatrix},\begin{pmatrix} -3\\ -1 \end{pmatrix}\right\}$$
(see Fig. \ref{exmpl_1_exmpl1_func_codif_exstr} b).

\begin{figure}[H]
\begin{minipage}[h]{0.49\linewidth}
\center{\includegraphics[width=0.5\linewidth]{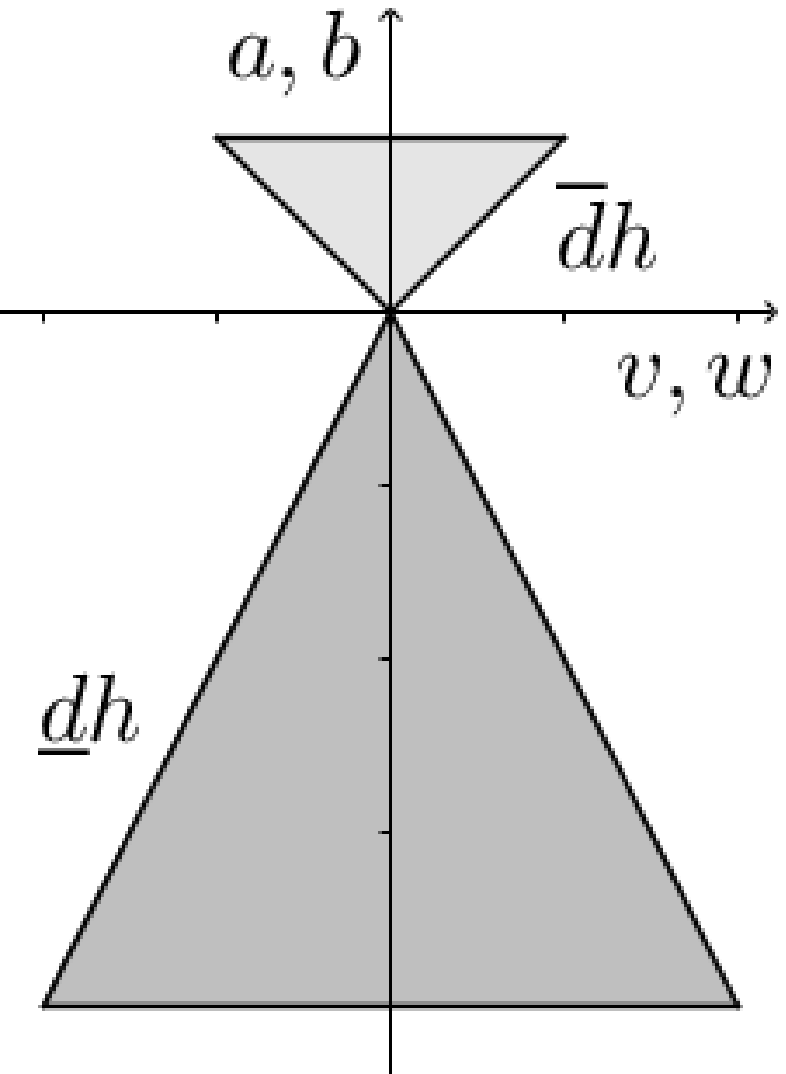}
\\ \textbf{a}}
\end{minipage}
\hfill
\begin{minipage}[h]{0.49\linewidth}
\center{\includegraphics[width=0.75\linewidth]{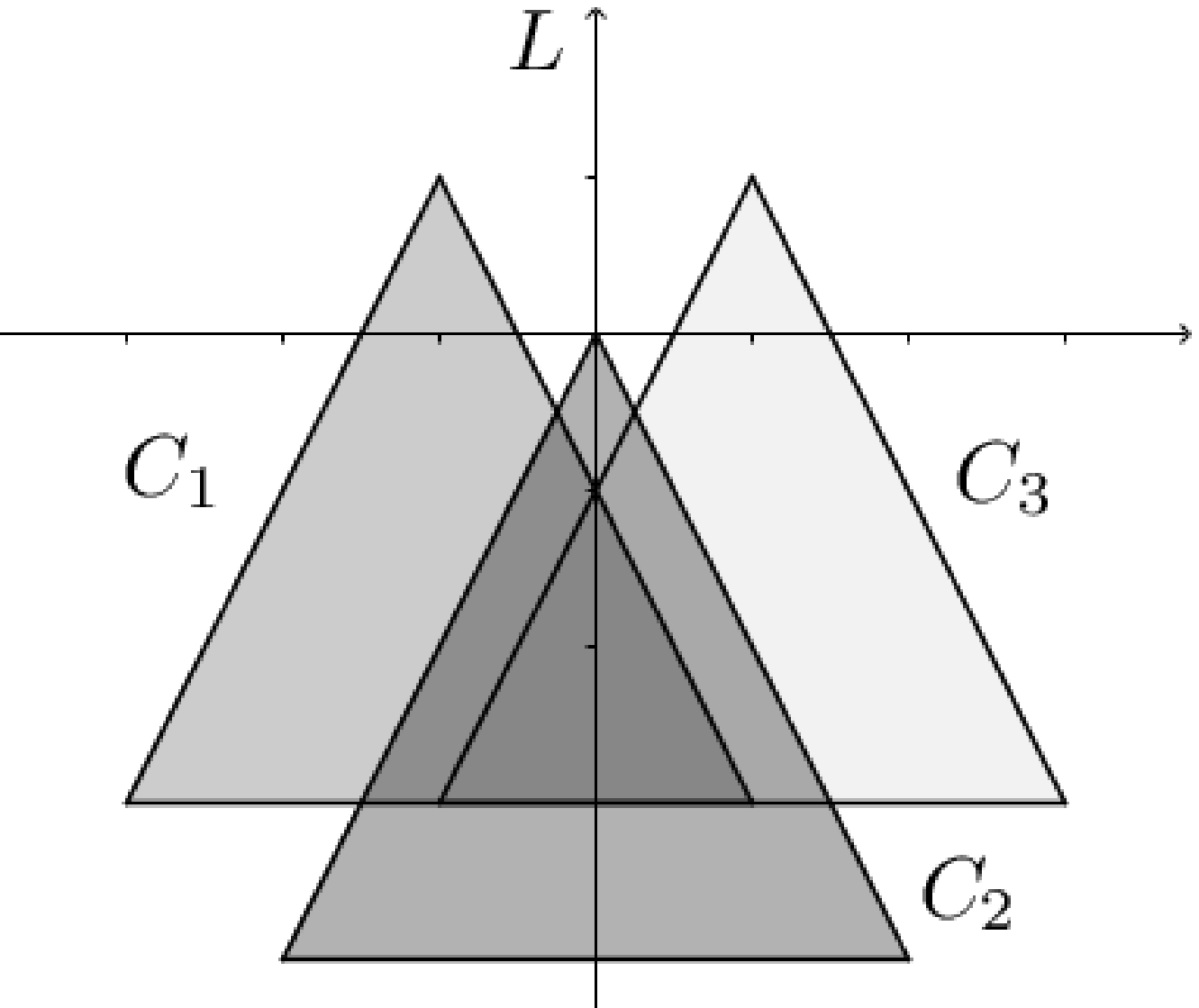}
\\ \textbf{b}}
\end{minipage}
\caption{\textbf{a.} A codifferential of the function $h$ in
Example 1.\ \textbf{b.} An upper coexhauster of the function $h$
in Example 1.} \label{exmpl_1_exmpl1_func_codif_exstr}
\end{figure}

Sets $\underline{d}h$ and $\overline{d}h$ are polyhedrons. We have
$v_1=2$, $v_2=0$, $v_3=-2$, $w_1=-1$, $w_2=0$, $w_3=1$. Therefore
$\operatorname{co}\{v_i\mid i=\overline{1,3}\}=[-2,2]$ and
$w_j\in[-2,2]$ for all $j=\overline{1,3}$. It is obvious (see Fig.
\ref{exmpl_1_exmpl1_func_codif_exstr} b) that $C_i\bigcap
L\neq\emptyset$ for all $i=\overline{1,3}$. This means that both
conditions (\ref{abbasov_bound_cond_codif_coex_equiv_eq1}) and
(\ref{abbasov_bound_cond_codif_coex_equiv_eq2}) are satisfied
here.

\end{example}

\section{Optimality conditions}\label{abbasov_codif_coex_DC_func_sect3}

Now let us proceed to the minimality conditions in terms of
coexhausters and codifferentials.

\begin{theorem}[Demaynov \cite{Demyanov-optimization-2012}]\label{Abbasov_th_Demaynov_min_cond}
For the inequality
$$h(\Delta)=\min_{C\in{\overline{E}}}\max_{[a,v]\in{C}}[a+\langle
v,\Delta\rangle]\geq 0$$ to be valid for all
$\Delta\in\mathbb{R}^n$ it is necessary and sufficient that the
condition
$$C\bigcap L_{+}\neq\emptyset\quad \forall C\in\overline{E},$$
where $L_0^+=\left\{(a,0_n)\mid a\geq 0\right\}$, holds.
\end{theorem}

\begin{theorem}[Polyakova \cite{Polyakova_JOGO_2011}]\label{Abbasov_th_Polyakova_min_cond}
For the inequality
$$h(\Delta)=\max_{i\in I}[a_i+\langle
v_i,\Delta\rangle]+\min_{j\in J}[b_j+\langle
w_j,\Delta\rangle]\geq 0$$ to be valid for all
$\Delta\in\mathbb{R}^n$ it is necessary and sufficient that the
condition
$$\operatorname{co}\{(a_i,v_i)\mid i\in I\}\bigcap \operatorname{co}\{(-b_j,-w_j),(0,-w_j)\}\neq\emptyset\quad\forall j\in J$$
holds.
\end{theorem}

Theorems \ref{Abbasov_th_Demaynov_min_cond} and
\ref{Abbasov_th_Polyakova_min_cond} can be used to show the
connection between polyhedral DC-functions, codifferentials and
coexhausters.

\begin{theorem}\label{abbasov_th_eq_min_cond}
Let the function $$h(\Delta)=\max_{i\in I}[a_i+\langle
v_i,\Delta\rangle]-\max_{j\in J}[b_j+\langle w_j,\Delta\rangle]$$
be given, where $I$ and $J$ are finite index sets. Then the
following statements are equivalent

\begin{enumerate}
\item The inequality $h(\Delta)\geq 0$ holds for all
$\Delta\in\mathbb{R}^n$.

\item The condition
\begin{equation}\label{abbasov_minimality_cond_codif_coex_equiv_eq2}
\operatorname{co}\{(a_i,v_i)\mid i\in I\}\bigcap
\operatorname{co}\{(b_j,w_j),(0,w_j)\}\neq \emptyset
\end{equation} holds for all $j\in J$.

\item The condition
\begin{equation}\label{abbasov_minimality_cond_codif_coex_equiv_eq3}
C\bigcap L^{+}\neq\emptyset
\end{equation}
holds for all $C\in\overline{E}$, where $L^+=\left\{(a,0_n)\mid
a\geq 0\right\}$ and
$$\overline{E}=\left\{C\mid C=\operatorname{co}\{[a_i-b_j,v_i-w_j],
i\in I\}, j\in J \right\}.$$
\end{enumerate}
\end{theorem}

\begin{proof}
We only need to prove the equivalence of conditions
(\ref{abbasov_minimality_cond_codif_coex_equiv_eq2}) and
(\ref{abbasov_minimality_cond_codif_coex_equiv_eq3}), since the
rest parts of the proof follows immediately from Theorems
\ref{Abbasov_th_Demaynov_min_cond} and
\ref{Abbasov_th_Polyakova_min_cond}.

First of all note that according to
(\ref{1904e3am_codiff_rewritten}) we have $a_i\leq 0$ and $b_i\leq
0$ for any $i\in I$ and $j\in J$.

Let condition (\ref{abbasov_minimality_cond_codif_coex_equiv_eq2})
be valid. Choose an arbitrary $j\in J$. Then there exists
$\lambda_{i}$, $i\in I$ such that
$$
\begin{cases}
\displaystyle\sum_{i\in I}\lambda_{i}=1,\\
\lambda_{ij}\geq 0\quad\forall i\in I,\\
\end{cases}
$$
for which we have
$$
\begin{cases}
\displaystyle\sum_{i\in I}\lambda_{i}a_i\geq b_j,\\
\displaystyle\sum_{i\in I}\lambda_{i}v_i=w_j,
\end{cases}
$$
whence
$$
\begin{cases}
\displaystyle\sum_{i\in I}\lambda_{i}(a_i-b_j)\geq 0,\\
\displaystyle\sum_{i\in I}\lambda_{i}(v_i-w_j)=0,
\end{cases}
$$
 This immediately brings us to
 (\ref{abbasov_minimality_cond_codif_coex_equiv_eq3}).

Since all the above steps of the proof can be reversed, we
conclude that (\ref{abbasov_minimality_cond_codif_coex_equiv_eq3})
implies (\ref{abbasov_minimality_cond_codif_coex_equiv_eq2}).

\qed
\end{proof}

\begin{remark}
Condition (\ref{abbasov_minimality_cond_codif_coex_equiv_eq2}) can
be rewritten in terms of codifferentials as
\begin{equation}\label{abbasov_codif_min_cond_remark2}
\underline{d}h\bigcap
\operatorname{co}\{(-b_j,-w_j),(0,-w_j)\}\neq \emptyset\quad
\forall j\in J,
\end{equation}
 where $\overline{d}h=\operatorname{co}\{(b_j,w_j)\mid j\in
J\}$.
\end{remark}

Similar result can be stated for maximum conditions.

\begin{theorem}\label{abbasov_th_eq_max_cond}
Let the function $$h(\Delta)=\max_{i\in I}[a_i+\langle
v_i,\Delta\rangle]-\max_{j\in J}[b_j+\langle w_j,\Delta\rangle]$$
be given, where $I$ and $J$ are finite index sets. Then the
following statements are equivalent

\begin{enumerate}
\item The inequality $h(\Delta)\leq 0$ holds for all
$\Delta\in\mathbb{R}^n$.

\item The condition
\begin{equation}\label{abbasov_maximality_cond_codif_coex_equiv_eq2}
\operatorname{co}\{(b_j,w_j)\mid j\in J\}\bigcap
\operatorname{co}\{(a_i,v_i),(0,v_i)\}\neq \emptyset
\end{equation} holds for all $i\in I$.

\item The condition
\begin{equation}\label{abbasov_maximality_cond_codif_coex_equiv_eq3}
C\bigcap L^{-}\neq\emptyset
\end{equation}
holds for all $C\in\underline{E}$, where $L^-=\left\{(a,0_n)\mid
a\leq 0\right\}$ and
$$\underline{E}=\left\{C\mid C=\operatorname{co}\{[a_i-b_j,v_i-w_j],
j\in J\}, i\in I \right\}.$$
\end{enumerate}
\end{theorem}

\begin{remark}
Condition (\ref{abbasov_maximality_cond_codif_coex_equiv_eq2}) can
be rewritten in terms of codifferentials as
\begin{equation}\label{abbasov_codif_max_cond_remark2}
\overline{d}h\bigcap \operatorname{co}\{(-a_i,-v_i),(0,-v_i)\}\neq
\emptyset\quad \forall i\in I,
\end{equation}
where $\underline{d}h=\operatorname{co}\{(a_i,w_i)\mid i\in I\}$.
\end{remark}

\begin{example}\label{abbasov_exmpl2} Let $f\colon\mathbb{R}\to\mathbb{R}$ be a
function of the form
$$f(x)=\max\{-x^2+2x,-x^2-2x,0\}-\max\{x-1,-x-1,0\}.$$

Fig. \ref{exmpl_2_exmpl2_func_graph} shows that the point $x_1=0$
is the local minimum of the function while the point $x_2=1$ is a
local maximum.

\begin{figure}[H]
  \includegraphics[width=0.45\textwidth]{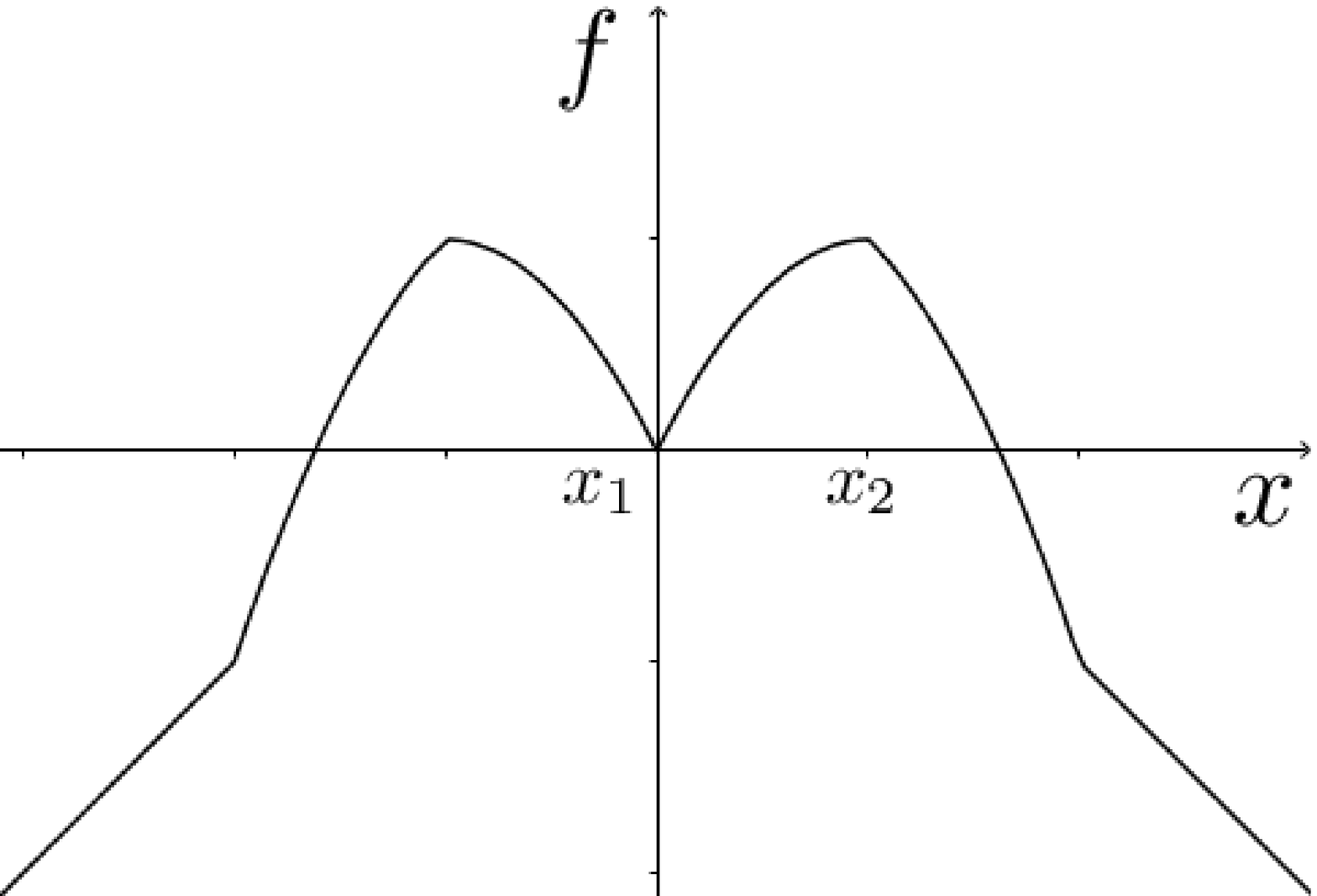}
\caption{Graph of the function $h$ in Example 2.}
\label{exmpl_2_exmpl2_func_graph}       
\end{figure}

Let us check conditions for a minimum at $x_1$ via Theorem
\ref{abbasov_th_eq_min_cond} and conditions for a maximum at $x_2$
via Theorem \ref{abbasov_th_eq_max_cond}. We start with an
expansion of $f$ in the neighborhood of $x$.
\begin{equation*}
\begin{split}
f(x+\Delta)=f(x)+\max\{-x^2+2x-f(x)+(-2x+2)\Delta,-x^2-2x\\-f(x)+(-2x-2)\Delta,-f(x)\}\\-\max\{x-1+\Delta,-x-1-\Delta,0\}+o(\Delta),
\end{split}
\end{equation*}
where $\displaystyle\lim_{\Delta\to 0}\frac{o(\Delta)}{\Delta}=0$.
Hence
$$\underline{d}h_x=\operatorname{co}\left\{\begin{pmatrix} -x^2+2x-f(x)\\ -2x+2 \end{pmatrix},\begin{pmatrix} -x^2-2x-f(x)\\ -2x-2 \end{pmatrix},\begin{pmatrix} -f(x)\\ 0 \end{pmatrix}\right\},$$
$$\overline{d}h_x=\operatorname{co}\left\{\begin{pmatrix} -x+1\\ -1 \end{pmatrix},\begin{pmatrix} x+1\\ 1 \end{pmatrix},\begin{pmatrix} 0\\ 0 \end{pmatrix}\right\}.$$
A codifferential at the point $x_1$ has the form (see Fig.
\ref{exmpl_2_exmpl1_func_codif_exstr} a)
$$\underline{d}h_{x_1}=\operatorname{co}\left\{\begin{pmatrix} &0\\ &2 \end{pmatrix},\begin{pmatrix} &0\\ -&2 \end{pmatrix},\begin{pmatrix} 0\\ 0 \end{pmatrix},\right\}, \quad \overline{d}h_{x_1}=\operatorname{co}\left\{\begin{pmatrix} &1\\ -&1 \end{pmatrix},\begin{pmatrix} 1\\ 1 \end{pmatrix},\begin{pmatrix} 0\\ 0 \end{pmatrix}\right\}.$$
For an upper coexhauster we have
$\overline{E}(x_1)=\{C_1,C_2,C_3\}$, where
$$C_1=\operatorname{co}\left\{\begin{pmatrix} 1\\ 1 \end{pmatrix},\begin{pmatrix} &1\\ -&3 \end{pmatrix},\begin{pmatrix} &1\\ -&1 \end{pmatrix}\right\},\quad C_2=\operatorname{co}\left\{\begin{pmatrix} 1\\
3\end{pmatrix},\begin{pmatrix} &1\\
-&1\end{pmatrix},\begin{pmatrix} 1\\ 1 \end{pmatrix}\right\},$$
$$C_3=\operatorname{co}\left\{\begin{pmatrix} 0\\ 2 \end{pmatrix},\begin{pmatrix} &0\\-&2 \end{pmatrix},\begin{pmatrix} 0\\ 0 \end{pmatrix}\right\}$$
(see Fig. \ref{exmpl_2_exmpl1_func_codif_exstr} b).

\begin{figure}[H]
\begin{minipage}[h]{0.49\linewidth}
\center{\includegraphics[width=0.80\linewidth]{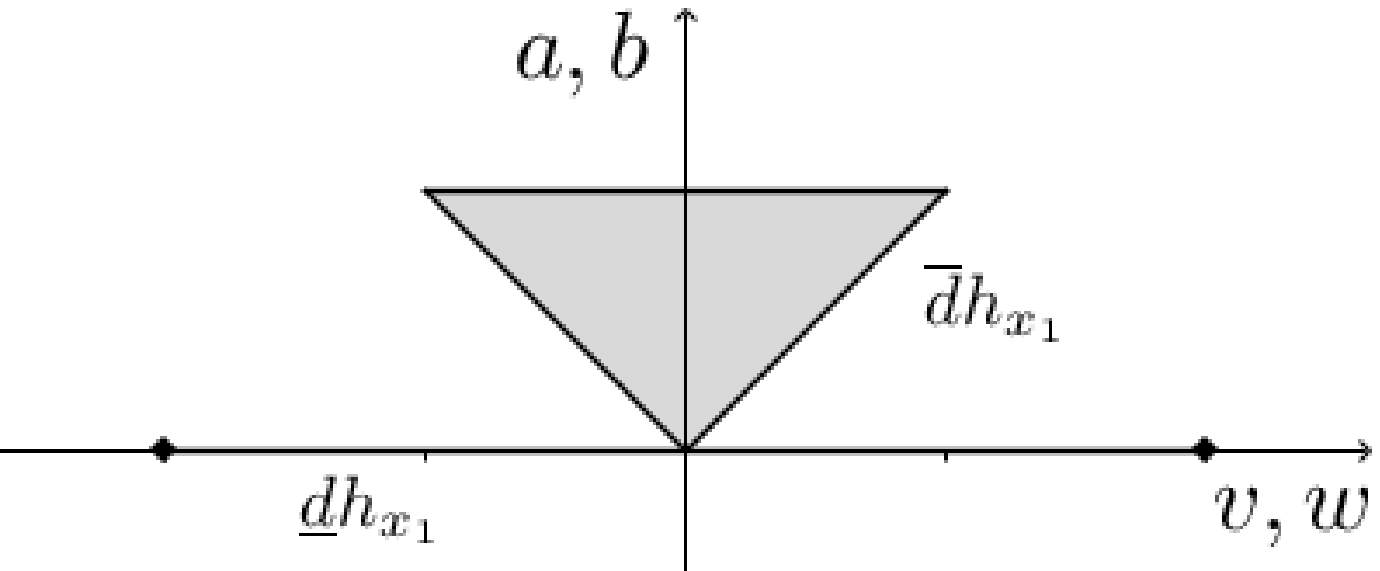}
\\ \textbf{a}}
\end{minipage}
\hfill
\begin{minipage}[h]{0.49\linewidth}
\center{\includegraphics[width=0.95\linewidth]{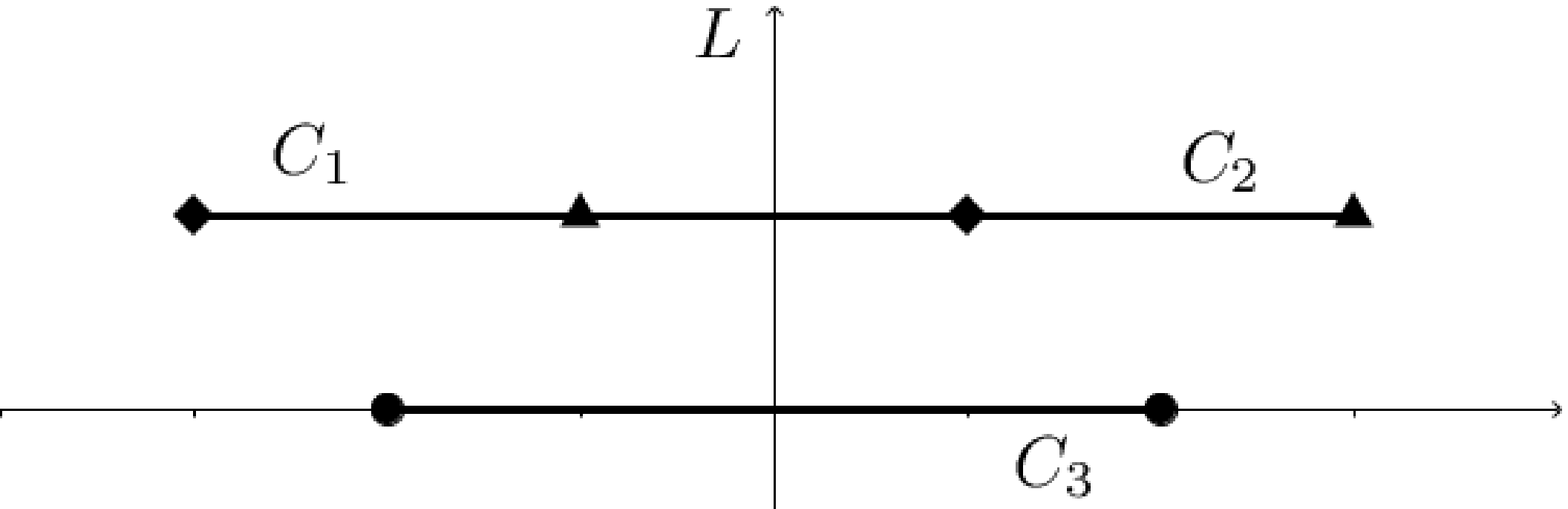}
\\ \textbf{b}}
\end{minipage}
\caption{\textbf{a.} A codifferential at point $x_1$ in Example
2.\ \textbf{b.} An upper coexhauster at point $x_1$ in Example 2.}
\label{exmpl_2_exmpl1_func_codif_exstr}
\end{figure}

Conditions (\ref{abbasov_minimality_cond_codif_coex_equiv_eq3})
and (\ref{abbasov_codif_min_cond_remark2}) hold which means that
$x_1$ is an $inf$-stationary point.

Now proceed to point $x_2$. A codifferential at this point has the
form
$$\underline{d}h_{x_2}=\operatorname{co}\left\{\begin{pmatrix} &0\\ &0 \end{pmatrix},\begin{pmatrix} -&4\\ -&4 \end{pmatrix},\begin{pmatrix} -&1\\ &0 \end{pmatrix},\right\}, \quad \overline{d}h_{x_2}=\operatorname{co}\left\{\begin{pmatrix} &0\\ -&1 \end{pmatrix},\begin{pmatrix} 2\\ 1 \end{pmatrix},\begin{pmatrix} 0\\ 0 \end{pmatrix}\right\},$$
(see Fig. \ref{exmpl_2_exmpl2_func_codif_exstr} a), whence for a
lower coexhauster we have $\underline{E}(x_2)=\{C_4,C_5,C_6\}$,
where
$$C_4=\operatorname{co}\left\{\begin{pmatrix} &0\\ -&1 \end{pmatrix},\begin{pmatrix} &2\\ &1 \end{pmatrix},\begin{pmatrix} &0\\ &0 \end{pmatrix}\right\},\quad C_5=\operatorname{co}\left\{\begin{pmatrix} -4\\
-5\end{pmatrix},\begin{pmatrix} -&2\\
-&3\end{pmatrix},\begin{pmatrix} -4\\ -4 \end{pmatrix}\right\},$$
$$C_6=\operatorname{co}\left\{\begin{pmatrix} -1\\ -1 \end{pmatrix},\begin{pmatrix} &1\\&1 \end{pmatrix},\begin{pmatrix} -&1\\ &0 \end{pmatrix}\right\}$$
(see Fig. \ref{exmpl_2_exmpl2_func_codif_exstr} b).

\begin{figure}[H]
\begin{minipage}[h]{0.49\linewidth}
\center{\includegraphics[width=0.9\linewidth]{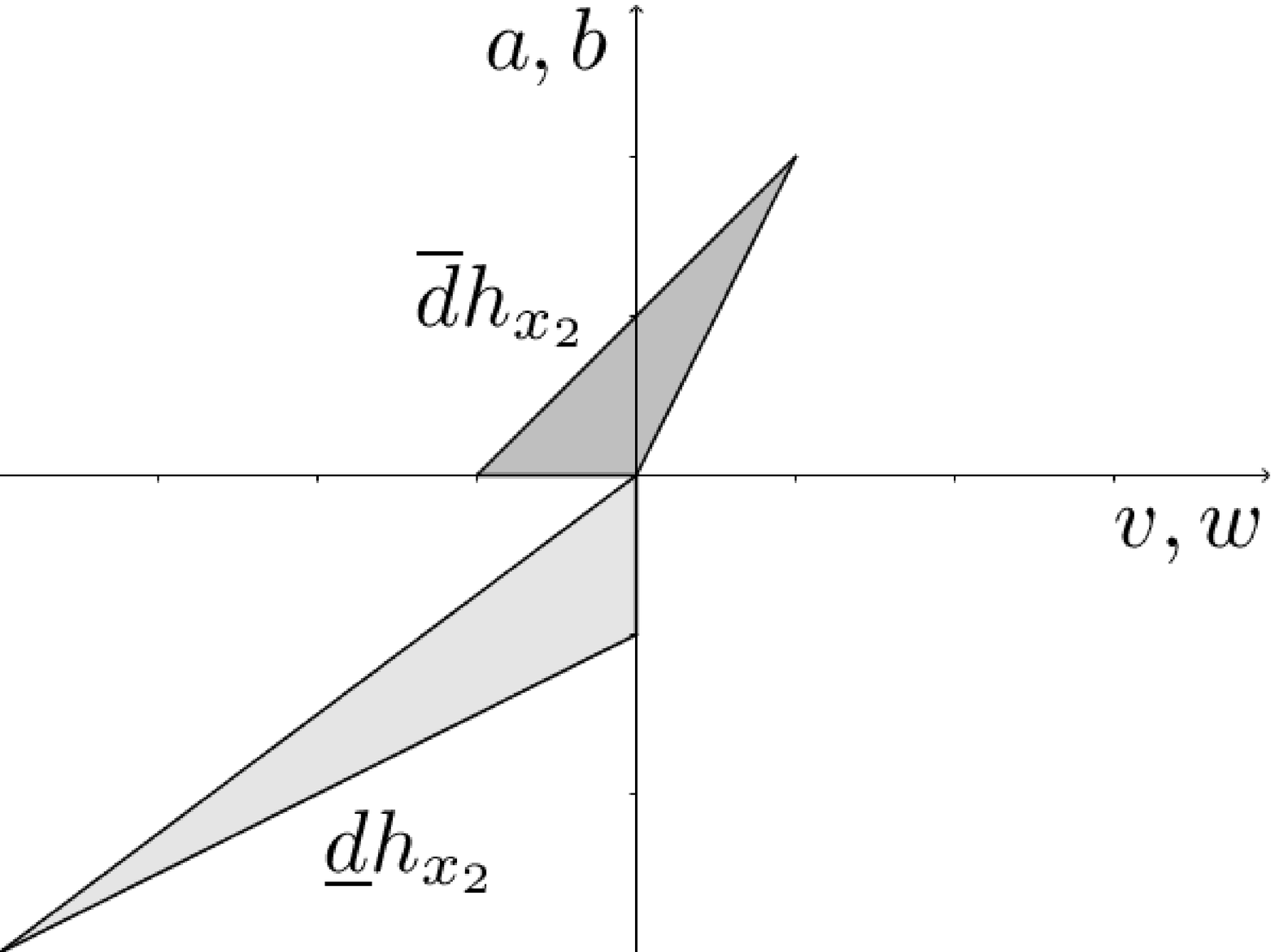}
\\ \textbf{a}}
\end{minipage}
\hfill
\begin{minipage}[h]{0.49\linewidth}
\center{\includegraphics[width=0.95\linewidth]{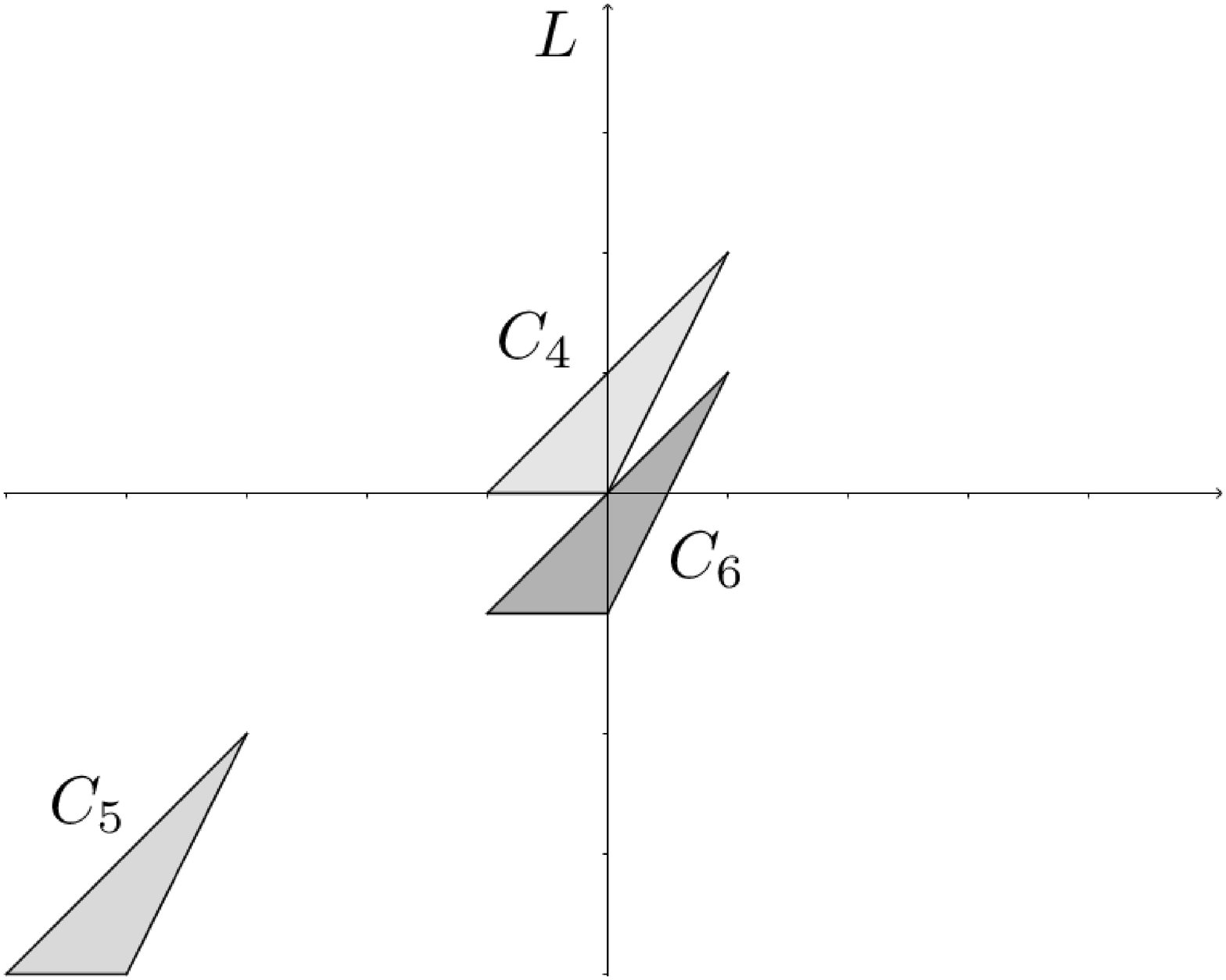}
\\ \textbf{b}}
\end{minipage}
\caption{\textbf{a.} A codifferential at point $x_2$ in Example
2.\ \textbf{b.} A lower coexhauster at point $x_2$ in Example 2.}
\label{exmpl_2_exmpl2_func_codif_exstr}
\end{figure}
Despite the fact that $x_2$ is a $sup$-stationary point we see
that neither condition
(\ref{abbasov_maximality_cond_codif_coex_equiv_eq3}) nor
 condition (\ref{abbasov_codif_max_cond_remark2}) are
fulfilled here.
\end{example}

\section{Conclusion}
We identified the connection between directional derivative and
nonhomogeneous approximations. Based on these connections, we
reformulated optimality conditions in terms of such
approximations.

Theorems that unite boundedness and optimality conditions in terms
of codifferentials, coexhausters and difference of polyhedral
convex functions were derived. It must be noted that in the case
of difference of polyhedral convex function $f$, expansion
(\ref{abbasov_expansion_main}) does not contain $o_x(\Delta)$
since this summand equals to zero and therefore all the conditions
described in Section \ref{abbasov_codif_coex_DC_func_sect3} are
necessary and sufficient conditions of global optimality. At the
same time, if we deal with a function which is not the difference
of polyhedral convex functions but can be approximated in that
form, results of Section \ref{abbasov_codif_coex_DC_func_sect3}
are only sufficient conditions of stationarity. This was
demonstrated at point $x_2$ in Example \ref{abbasov_exmpl2}.

The intention of this paper is to widen the facilities of
researchers in solving nondifferentiable optimization problems and
to make closer specialist working in different branches of
nonsmooth analysis.

\section*{Acknowledgements}
Results in Section 3 were obtained in the Institute for Problems
in Mechanical Engineering of the Russian Academy of Sciences with
the support of Russian Science Foundation (RSF), project No.
20-71-10032.

\end{document}